\documentclass[11pt]{article}
\usepackage[english]{babel}
\usepackage{amsmath} 	
\usepackage{amsfonts}
\usepackage{amsthm}
\usepackage{amssymb}
\usepackage{natbib}
\usepackage[utf8]{inputenc}

\usepackage{comment}
\usepackage{enumerate}
\usepackage[normalem]{ulem} 
\usepackage{chngcntr}
\counterwithout{paragraph}{subsubsection} 
\usepackage{lettrine} 
\usepackage{paralist} 
\usepackage[hmarginratio=1:1,top=32mm,columnsep=20pt]{geometry} 

\theoremstyle{plain}
\newtheorem{theorem}{Theorem}[paragraph]
\newtheorem{proposition}[theorem]{Proposition}
\newtheorem{lemma}[theorem]{Lemma}
\newtheorem{corollary}[theorem]{Corollary}
\theoremstyle{definition}

\newtheorem{example}[theorem]{Example}
\theoremstyle{remark}

\def\g{\boldsymbol}
\def\pr{{\rm Pr}}
\def\ci{\perp\!\!\!\perp}
\def\aci{\tilde \ci}

\begin{document}
\thispagestyle{empty}
\begin{center}
{\LARGE \textbf{One-Component Regular Variation \\ and Graphical Modeling of Extremes}\\}
\phantom{adsf}
{\large Adrien Hitz and Robin Evans\\}
{\large \textit{University of Oxford}\\}
\end{center}


\small
\setlength{\leftskip}{1cm}
\setlength{\rightskip}{1cm}

The problem of inferring the distribution of a random vector given that its norm is large requires modeling a homogeneous limiting density. We suggest an approach based on graphical models which is suitable for high-dimensional vectors.

We introduce the notion of one-component regular variation to describe a function that is regularly varying in its first component. We extend the representation and Karamata’s theorem to one-component regularly varying functions, probability distributions and densities, and explain why these results are fundamental in multivariate extreme-value theory. We then generalize Hammersley-Clifford theorem to relate asymptotic conditional independence to a factorization of the limiting density, and use it to model multivariate tails. 

\setlength{\leftskip}{0pt}
\setlength{\rightskip}{0cm}
\normalsize

\paragraph{Introduction}\label{par:intro} Regular variation describes the behavior of some functions when evaluated close to infinity \cite{Seneta1973}. In the multivariate case, regular variation of a measurable function $u$ on $\mathbb R^d\setminus\{\g 0\}$ with limit $v$ such that $u(\lambda \g 1),\, v(\lambda \g 1)>0,$ $\forall \lambda>0,$ is defined as
\begin{align}\label{eq:rv}
{u(t\g x)\over u(t\g 1)}\underset{t\rightarrow\infty}{\longrightarrow} v(\g x),\quad \forall \g x,
\end{align}
\cite{Resnick1987Extreme}. A striking result is that, in this case, the limit is homogeneous, i.e., there exists $\alpha\in\mathbb R$ such that $v(\lambda \g x)=\lambda^\alpha v(\g x),$ $\forall \lambda>0,$ $\forall \g x.$ 

We suggest an alternative definition: a function $u$ is called {\em regularly varying w.r.t.\ its first component\/} if
 $${ u(t  x,\g y)\over  u(t,\g 1)}\rightarrow v(x,\g y),$$ 
where $u$ and $v$ are measurable non-negative functions on $(0,\infty) \times \mathbb R^{d-1}$ such that $u(\cdot,\g 1)>0$ and $v(\cdot,\g 1)>0.$ One-component regular variation includes regular variation as a special case. For instance, $u$ is regularly varying on $\mathbb R^d_+\setminus\{\g 0\}$ if and only if $u\circ \varphi^{-1}$ is regularly varying in its first component, where $\varphi(\g x)=(||\g x||,\g x/||\g x||)$ and $\varphi^{-1}(r,\g\theta)=r\g\theta.$ In Section~\ref{par:ocrvFunc}, we generalize two important results in Karamata's theory --- i.e.\ the study of regular variation --- from the univariate to the multivariate case. The first is the representation theorem \cite{Bingham1989}, which states that any one-component regularly varying function can be written in a specific form. The second, Karamata's theorem, reveals relations between a one-component regularly varying function and its integral.  

Regular variation is, above all, the limit of a fraction. The specific form of the limit $v(x)=x^\alpha$ in the univariate case comes from the fact that the operation $\alpha\cdot x\mapsto x^\alpha$ is distributive over the division; when regular variation is applied in probability, the fraction corresponds to conditioning. The multiplication $t x,$ however, is an arbitrary choice of operation that can be replaced by a more general scaling \cite{Bingham2010}. Choosing $t\star x\mapsto T^{-1}\{T(t)T(x)\}$ for a diffeomorphism $T$ gives a limit of the form $T(x)^\alpha,$ thus extending regular variation to other decays than the power law. This is not just a trivial transformation, but enriches the representation and Karamata's theorems.

Extreme-value theory studies the distribution of a random vector $\g X$ in regions far from the origin. In \cite{BasrakPhD2000}, $\g X$ is called regularly varying if there exists a probability distribution $\nu$ on $C_{||\cdot||}=\{\g x\in \mathbb R^d: ||\g x|| \geq  1\}$ s.t.\ $\nu(\lambda C_{||\cdot|| })>0$ for some $\lambda>1,$ and
\begin{align}\label{eq:mrvVector}
t^{-1}\g X\mid ||\g X||\geq t  \overset{w}{\Rightarrow} \nu,
\end{align}
where $\overset{w}{\Rightarrow}$ denotes weak convergence. In this case, $\nu$ is homogeneous, i.e., there exists $\alpha>0$ such that $\nu(\lambda A)=\lambda^{-\alpha} \nu(A),$ for all Borel sets $A$ and $\lambda\geq 1.$ Basrak's exact definition involves vague convergence to a non-null Radon measure $\tilde \nu$ on $( \{-\infty\}\cup \mathbb R\cup \{\infty\})^d\setminus\{\g 0\}$ called the exponent measure. For simplicity and without losing generality, we stick to weak convergence between probability distributions throughout the paper. 

We say that a random vector $(X,\boldsymbol Y)$ is {\em regularly varying w.r.t.\ its first component,\/} with a limiting probability distribution $\mu$ on $[1,\infty)\times \mathbb R^{d-1}$ if 
$$(t^{-1}X,\,\boldsymbol Y) \mid X\geq t \overset{w}{\Rightarrow} \mu,$$
and $\mu_X\neq \delta_1,$ the Dirac mass on $1.$ Limits of random vectors with an extreme component have already been studied \cite{Heffernan2007,Resnick2014} and it is known that $\mu=P_\alpha \times H,$ where $P_\alpha$ is the Pareto distribution of index $\alpha>0$ and $H$ is a probability distribution on $\mathbb R^{d-1}.$ In particular, (\ref{eq:mrvVector}) is equivalent to one-component regular variation of $\varphi(\g X),$ in which case $H$ corresponds to the angular distribution \cite{Beirlant2004}.

In Section~\ref{par:ocrvDistr}, we adapt the representation and Karamata's theorems for random variables and, as a result, characterize homogeneous probability measures. In addition, we show that if $\g X$ admits a regularly varying probability density with non-null limit and if the sequence is for instance dominated by an integrable function, then 
\begin{align}\label{eq:convDensityIntro}
f_{t^{-1}\g X\,\mid\, ||\g X|| \geq t}\rightarrow f_{\g Y},
\end{align}
for a probability density $f_{\g Y}$ on $C_{||\cdot||},$ and thus $\g X$ is regularly varying. In this case, $f_{\g Y}$ is homogeneous of order $-\alpha-d.$ This extends Theorem 2.1 in \cite{de1987regular}, which implies regular variation of $\g X$ when, in addition to regular variation of $f_{\g X},$ the convergence is uniform, and the limit is bounded on $C_{||\cdot||}.$

Besides describing exceedances over a threshold, regular variation of $\g X$ is directly related to the behavior of maxima of independent copies $\g X^{(i)}.$ As explained in \cite{Resnick1987Extreme}, if $\g X$ is non-negative, then it is regularly varying if and only if there exists a sequence $a_n\rightarrow \infty$ such that $\max_{i=1,\ldots,n} a_n^{-1} \g X^{(i)} \overset{w}{\Rightarrow} G$ for a non-degenerate probability distribution $G$ on $\mathbb R_+^d;$ if the marginals of $G$ are not degenerated, then they are Fréchet distributed and $G$ is called Type II multivariate extreme-value distribution. The cumulative distribution function of $G$ is $\exp\{-\tilde \nu([\g 0,\g x]^c \},$ where $\tilde \nu$ is the exponent measure. The problem of modeling $\tilde \nu$ or the corresponding copula has received much attention in multivariate analysis. Unless $d$ is particularly small, the density of $G$ is approximated by the one of $\tilde \nu$ using a Taylor expansion, i.e., up to a constant, by $f_{\g Y}.$ Hence, modeling both maxima and exceedances of $\g X$ requires modeling a homogeneous density. 

The usual strategy is to write $f_{\g Y}(\g y)=||\g y||^{-\alpha-d} h(\g y/||\g y||)$ and model the angular density $h=dH$ satisfying some constraints to ensure that $\g Y$ has Pareto marginals. Suggested models for $h$ include the asymmetric logistic \cite{Tawn90}, the pairwise beta \cite{CooleyDavisNaveau}, its generalization \cite{BallaniSchlather}, mixtures of Dirichlet \cite{BoldiDavison} and the angular density of the Hüsler-Reiss exponent measure \cite{Engelke2015}.  However, they suffer limitations in high-dimensions as the number of parameters explodes and may lack flexibility to describe multivariate tails.


As a consequence of (\ref{eq:convDensityIntro}), any parametric probability density $f_{\g X}(\cdot;\g \theta)$ that is regularly varying with decay $T$ and satisfies the additional condition induces a homogeneous parametric density $f_{\g Y}(\cdot\,;\g\theta)$ in the limit. If the multivariate marginals and the censored densities of $f_{\g X}(\cdot;\g \theta)$ are available, we directly find the ones of $f_{\g Y}(\cdot;\g \theta)$ by passing to the limit. A complementary approach pursued from Section~\ref{par:me} is to simplify $f_{\g Y}$ using graphical models, which have been successful in reducing dimensionality \cite{wainwright2008graphical}. In Section~\ref{par:agm}, we define a new notion called {\em asymptotic conditional independence\/} and we generalize Hammersley-Clifford for sequence of densities that factorize in the limit w.r.t.\ a graph. 

Section~\ref{par:agme} explains how to use asymptotic graphical models for multivariate exceedances. We consider a sequence $\g X_t^C$ based on censoring the marginals of $\g X$ whose absolute values fall below $t.$ Under (\ref{eq:convDensityIntro}) and $f_{\g Y}>0,$ the generalized Hammersley-Clifford theorem states that some marginals of $\g X_{t}^C$ are asymptotically conditionally independent given the rest of the vector if and only if the censored density of $\g Y$ is determined by lower dimensional marginals $\g Y_C,$ where $C$ are the cliques of a graph. We show that $\g Y_C$ corresponds to the limit of $\g X_C,$ enabling inference to be performed on each clique separately once the graph has been selected. 


\section{One-Component Regular Variation for Functions}\label{par:ocrvFunc}  Consider the commutative and associative operation 
\begin{align}\label{eq:defStar}
x\star y:=T^{-1}\{T(x)T(y)\},\quad x,y\in E=[e_0,e_1),
\end{align}
where $T:E\rightarrow [1,\infty)$ is a diffeomorphism, for $e_0\in \mathbb R,$ $e_1\in \mathbb R\cup \{\infty\}.$ When $E=[0,\infty),$ possible operations include multiplication ($T\equiv\text{id}$), addition ($T\equiv\exp$) and $x\star y=||x,y||_p$ for $p>0,$ $T(x)=\exp(x^p).$ When $E=[0,e_1)$ and $e_1<\infty,$ $T(x)=(1-e_1^{-1} x)^{-1}$ gives $x\star y=x+y-e_1^{-1} xy.$ We call a positive and measurable function $u$ on $E$ {\em regularly varying with decay $T$\/} and limit $v$ if
\begin{align}\label{eq:T-RV}
{u(t \star x)\over u(t\star e_0)}\rightarrow v(x)>0,
\end{align}
where the arrow stands for pointwise convergence on the entire domain of definition as $t \uparrow e_1.$ It follows that $u\circ T^{-1}$ is regularly varying, and thus $v(x)=T(x)^\alpha$ for some $\alpha \in \mathbb R.$ The convergence moreover is uniform \cite{Resnick1987Extreme}. We denote (\ref{eq:T-RV}) by $u\in T\text{-RV}_{\alpha}.$ 

Let us extend this notion to the multivariate setting. Consider two measurable and non-negative functions $u$ and $v$ on $E\times \mathbb R^{d-1}$ such that $u(\cdot,\g 1)>0,$ $v(\cdot,\g 1)>0,$ and a non-negative function $h$ on $\mathbb R^{d-1}$ satisfying $h(\g 1)=1.$ We call $\g 1$ the pivot and its choice is arbitrary. 

\begin{lemma}[Characterization]\label{le:equivRVx}
The following are equivalent.
\begin{enumerate}[(i)]
\item \label{item:charRVxRV} ${ u(t \star  x,\g y)/  u(t\star e_0,\g 1)}\rightarrow v(x,\g y),$
\item  \label{item:charRVxCond1}  $u(\cdot,\g 1)\in T\text{-RV}_{\alpha}$ and ${u(t,\g y)/ u(t, \g 1)}\rightarrow h(\g y),$ 
\item \label{item:charRVxCondY}  $u(\cdot,\g y)\in T\text{-RV}_{\alpha},$ $\forall \g y$ s.t.\ $h(\g y)>0,$ and ${u(t,\g y)/ u(t,\g 1)}\rightarrow h(\g y),$ 
\item \label{item:charRVxAlphaH}  ${ u(t \star x,\g y)/  u(t\star e_0,\g 1)}\rightarrow  T(x)^{\alpha}\;h(\g y)$ uniformly in $x.$ 
\end{enumerate}
\end{lemma}

\begin{proof} 
 (\ref{item:charRVxRV}) $\Rightarrow$
 (\ref{item:charRVxCond1}): to derive the first condition, set $\g y=\g 1;$ for the second, $x=e_0.$
 (\ref{item:charRVxCond1})$\Rightarrow$(\ref{item:charRVxCondY}): let $\g y$ s.t.\ $h(\g y)>0.$ Since $u(t\star e_0,\g y)/u(t\star e_0,\g 1)\rightarrow h(\g y),$ $u(t\star e_0,\g y)>0$ for $t$ sufficiently large. It follows that 
$${u(t\star x,\g y)\over  u(t\star e_0,\g y)}={u(t\star x,\g y)\over u(t\star x,\g 1)}{u(t\star x,\g 1)\over u(t\star e_0,\g 1)}{u(t\star e_0,\g 1)\over u(t\star e_0,\g y)}\rightarrow T(x)^\alpha.$$
 (\ref{item:charRVxCondY}) $\Rightarrow$ (\ref{item:charRVxAlphaH}): as $h(\g 1)>0,$
 $${u(t\star x,\g y)\over  u(t\star e_0,\g 1)}={u(t\star x,\g y)\over u(t\star x,\g 1)} {u(t\star x,\g 1)\over u(t\star e_0,\g 1)}\rightarrow T(x)^{\alpha}h(\g y),$$
 and the convergence is uniform in $x.$ 
 (\ref{item:charRVxAlphaH}) $\Rightarrow$ (\ref{item:charRVxRV}): clear. 
\end{proof}

 We say that $u$ is {\em regularly varying in its first component} if (\ref{item:charRVxRV})--(\ref{item:charRVxAlphaH}) hold, written $u\in T\text{-RV}^x_{\alpha}(h),$ or simply $\text{RV}^x_{\alpha}(h)$ when $T\equiv \text{id}.$ Condition $u(\cdot,\g y)\in\text{RV}_{\alpha},$ $\forall \g y,$ corresponds to the uniform regular variation of \cite{Meerschaert1993} if it is assumed further that the convergence is uniform in $\g y.$ 

We now generalize the representation theorem for univariate regular variation (see \cite{Bingham1989} in the case $T\equiv \text{id},$ \cite{Jaros} otherwise).

\begin{theorem}[Multivariate Representation Theorem] \label{thm:reprRVx}
 It holds $u\in T\text{-RV}^x_{\alpha}(h)$ if and only if
\begin{align*}
u(x,\g y)=c(x) \exp\left\{\int_{e_0}^x \alpha(z) {T'(z)\over T(z)}  dz \right\} q(x,\g y), 
\end{align*}
for some measurable functions s.t.\ $c(t)\rightarrow c>0,$ $\alpha(t)\rightarrow \alpha,$ $q(t,\g y)\rightarrow h(\g y)$ as $t\uparrow e_1.$ 
\end{theorem}

As a consequence, $\forall \epsilon>0,$ $\exists c_1,c_2>0$ such that 
\begin{align}\label{eq:boundT}
c_1 T(x)^{\alpha-\epsilon} h(\g y) < u(x,\g y)< c_2 T(x)^{\alpha+\epsilon} h(\g y),
\end{align} 
for $x$ sufficiently large and $\g y$ satisfying $h(\g y)\neq 0.$

\begin{proof} For the direct implication, write $u(x,\g y)=u(x,\g y)/u(x,\g 1)\, u(x,\g 1).$ Lemma~\ref{le:equivRVx} gives $q(t,\g y):=u(t,\g y)/u(t,\g 1)\rightarrow h(\g y)$ and $u(x ,\g 1)\in T\text{-RV}_{\alpha}.$ The conclusion follows by applying the representation theorem on the regularly varying function $u\{T^{-1}(x),\g 1\}.$ For the reverse, consider $z=t\star\bar z$ and $dz=T(t)T'(\bar z)/T'(t\star \bar z)$ to find 
$$\int_{t\star e_0}^{t\star x}\alpha(z){T'(z)/ T(z)} dz=\int_{e_0}^{x}\alpha(t\star \bar z){T'(\bar z)/ T(\bar z)} d\bar z\rightarrow \alpha \log T(x),$$
and thus $u(t\star x,\g y)/u(t\star e_0,\g 1)\rightarrow T(x)^\alpha h(\g y).$
\end{proof}

Another important result is Karamata's theorem, which relates the regular variation of a univariate function to that of its integral \cite{Bingham1989}. We generalize it for one-component regular variation, only treating the case of a negative power index $-\alpha$ for $\alpha>0.$ Suppose that $\bar U(x,\g y):= \int_{x}^\infty u(z,\g y)dz$ exists $\forall x,\g y$ (it exists for $e_0$ sufficiently large if $u(x,\g y)/T'(x)$ $\in T\text{-RV}^x_{-\alpha-1}(h)$ as a consequence of (\ref{eq:boundT})). 

\begin{theorem}[Multivariate Karamata's Theorem]\label{thm:karamataRVx}
It holds 
\begin{align*}
{u(x,\g y)\over T'(x)} \in T\text{-RV}^x_{-\alpha-1}(h)\quad \iff \quad {T(t)\over T'(t)} {u(t, \g y)\over  \bar U(t,\g 1)} \rightarrow \alpha\,  h(\g y).
\end{align*}
 In this case, $\bar U\in T\text{-RV}^x_{-\alpha}(h),$ and its representation has coefficient $c(x)\equiv \bar U(e_0,\g 1).$ 
\end{theorem}

\begin{proof} 
To prove the direct implication, use the change of variable $z=t\star \tilde z$ to find
\begin{align*}
{T(t\star e_0)\over T'(t\star e_0)}{u(t\star e_0,\g y)\over \int_{t\star e_0}^{\infty} u(z,\g 1)dz} & = 
  \left\{ \int_{e_0}^\infty {u(t\star \tilde z,\g 1)\over u(t\star e_0,\g 1)}{T'(t\star e_0)\over T'(t\star \tilde  z)} 
T'(z)
d \tilde z\right\}^{-1} {u(t\star e_0,\g y)\over u(t\star e_0,\g 1)}\\
& \rightarrow  \left\{\int_{e_0}^\infty  T(\tilde z)^{-\alpha-1} T'(\tilde z)d  \tilde z\right\}^{-1} h(\g y) =\alpha h(\g y) ,
\end{align*} 
because $u(\cdot,\g 1)/T'(\cdot)\in T\text{-RV}_{-\alpha-1}$ and $u(t\star e_0,\g y)/u(t\star e_0,\g 1)\rightarrow h(\g y)$ from Lemma~\ref{le:equivRVx}. Limit and integral can be exchanged because $c T(z)^{-\alpha-1+\epsilon}T'(z)$ dominates the integrand for some $c>0,$ $\epsilon\in (0,\alpha)$ and $t$ sufficiently large thanks to (\ref{eq:boundT}).

For the reverse implication, let $\alpha(x):=$ $\{u(x,\g 1)/\bar U(x,\g 1) \}/$ $\{T'(x)/T(x)\},$ which satisfies $\alpha(t)\rightarrow \alpha.$ Integrate $u(z,\g 1)/\bar U(z,\g 1)=$ $\alpha(z) T'(z)/T(z)$ from both sides between $e_0$ and $x$ to obtain 
$\bar U(x,\g 1) = \bar U(e_0,\g 1) \exp\left\{- \int_{e_0}^x \alpha(z) {T'(z)/ T(z)} dz \right\}.$
This is the representation of a $T$-regularly varying function with $c(x)=\bar U(e_0,\g 1)$ thanks to Theorem \ref{thm:reprRVx}. Hence, ${T'(t\star e_0)/ T'(t\star x)}\, {u(t\star x,\g y)/ u(t\star e_0,\g 1)}$ equals
\begin{multline*}
 {T'(t\star e_0) \, \bar U(t\star e_0,\g 1)\over T(t\star e_0)\, u(t\star e_0,\g 1)}\, {T(t\star x) \, u(t\star x,\g y)\over T'(t\star x)\,\bar U(t\star x,\g 1) }\,{T(t\star e_0)\over T(t\star x)}  \,{\bar U(t\star x,\g 1)\over \bar U(t\star e_0,\g 1)} 
\rightarrow  T(x)^{-\alpha-1} h(\g y),
\end{multline*}
which ends the proof of the equivalence in (\ref{thm:karamataRVx}). Moreover, $\bar U \in T\text{-RV}_{\alpha}(h)$ because
\begin{multline*}
{\bar U(t\star x,\g y)\over \bar U(t\star e_0,\g 1)} = \bar U(t\star e_0,\g 1)^{-1} \,\int_{t\star x}^\infty u(\tilde x,\g y)d\tilde x \\
= {u(t\star e_0,\g 1)\over \bar U(t\star e_0,\g 1)}{T(t\star e_0)\over T'(t\star e_0)}\; \int_x^\infty {u(t\star z,\g y)\over u(t\star e_0,\g 1)}{T'(t\star e_0)\over T'(t\star z)} T'(z)dz  \rightarrow  T(x)^{-\alpha} h(\g y).
\end{multline*}
$ $
\end{proof}

For instance, consider the Gaussian distribution $\mathcal N(\mu,\sigma^2)$ with probability density $\phi$ and survival function $\bar \Phi.$ Since $2^{-1}t^{-1} \phi(t)/ \bar \Phi(t) \rightarrow  2^{-1} \sigma^{-2},$ Theorem \ref{thm:karamataRVx} gives $\phi(x)/T'(x)\in T\text{-RV}_{-2^{-1}\sigma^{-2}-1}$ and $\bar \Phi\in T \text{-RV}_{-2^{-1}\sigma^{-2}}$ for $T(x)=\exp(x^2).$ Let $f$ be the probability density of the log-Cauchy distribution and $\bar F$ its survival function. As $t\log t \,f(t)/\bar F(t)$ $\rightarrow$ $1,$ it follows that $x f(x) \in \log\text{-RV}_{-2},$ and $\bar F\in \log\text{-RV}_{-1}.$

We now extend the standard multivariate regular variation in (\ref{eq:rv}) to general decays. We denote the sign of $x\in\mathbb R$ by $\sigma_x\in \{-1,0,1\}$ and $\g\sigma_{\g x}=(\sigma_{x_1},\ldots, \sigma_{x_d})$ when $\g x\in \mathbb R^d.$ Operations between vectors are done componentwise, $x\star \g y:=(x\g 1) \star \g y,$ and $\g e_0:=e_0\g 1.$ We say that $u:\mathbb R^d\rightarrow [0,\infty)$ satisfying $u(\lambda \g 1)>0,$ $\forall \lambda\in E,$ is {\em regularly varying with decay $T$\/} if there exists $v$ such that $v(\lambda \g 1)>0,$ $\forall \lambda\in E,$ and
\begin{align}\label{eq:TRV}
{u\{( t \star |\g x|)\g \sigma_{\g x} \}\over u( t \star \g e_0) }\rightarrow v(\g x),
\end{align}
on $F:=\{(-e_1,-e_0)\cup \{0\}\cup [e_0,e_1)\}^d\setminus\{\g 0\}$; by convention, $(t\star |x|) 1_{x}=0$ if $x=0.$ 

Multivariate regular variation is easily expressed in terms of one-component regular variation by introducing the following change of variable. We say that $\phi:C\rightarrow (0,\infty)\times \Omega,$ for $\Omega\subset \mathbb R^{d-1},$ defines a {\em radial system of coordinates\/} if it has the form
\begin{align}\label{eq:defnRadialCoord}
\phi: \g x\mapsto \{r(\g x),\g\theta(\g x)\},\quad \phi^{-1}:(r,\g \theta)\mapsto \{r\star |\g\theta^{-1}(\g\theta)|\}\g\sigma_{\g \theta^{-1}(\g \theta)},
\end{align}
where $r(\cdot)$ and $\g\theta(\cdot)$ satisfy $r\{(\lambda\star |\g x|)\g \sigma_{\g x}\}=\lambda \star r(\g x),$ $\g \theta\{(\lambda\star |\g x|)\g \sigma_{\g x}\}=\g\theta(\g x),$ $\forall \lambda\in E.$ Examples include spherical coordinates on $\mathbb R^d$ and pseudo-polar coordinates defined by $r(\g x)=||\g x||,$ a norm, and $\g\theta (\g x)=\g x_{1:d-1}/r(\g x)$ on $\mathbb R^d_+.$ When $T\equiv \exp,$ the latter translates into $r(\g x)=\log(||e^{\g x}||),$ $\g\theta(\g x)=\g x - r(\g x).$ It becomes clear that $u$ is regularly varying if and only if $g(r,\g \theta):=u\{\phi^{-1}(r,\g\theta)\}$ is regularly varying w.r.t.\ its first component since
\begin{align}\label{eq:gtrt}
{g(t \star r,\g\theta)\over g(t\star  \g 1_r,\g 1_{\g \theta})}={u\{(t\star |\g x|)\g\sigma_{\g x}\}\over u(t\star \g e_0)} \rightarrow v(\g x), 
\end{align}
and in this case, Lemma~\ref{le:equivRVx} gives $v(\g x)=T\{r(\g x)\}^\alpha h\{\g\theta(\g x)\},$ where $\g 1_r=r(\g e_0)$ and $\g 1_{\g \theta} =\g\theta(\g e_0).$ In particular, $v$ is homogeneous of order $\alpha,$ i.e., $v\{(\lambda\star |\g x|)\g\sigma_{\g x}\}=T(\lambda)^\alpha v(\g x),$ $\forall \lambda \in E,$ and we write (\ref{eq:TRV}) as $u\in T\text{-RV}_{\alpha}(v).$

\section{One-Component Regular Variation for Probability Distributions}\label{par:ocrvDistr} 
So far, one-component regular variation has been treated for functions; we develop it further for distributions. A specific representation is found for regularly varying random vectors. Their limits moreover form the class of distributions that are homogeneous w.r.t.\ their first component. In a subsequent result, we show that one-component regular variation of the probability density implies, under an extra condition, one-component regular variation of the distribution. 

For simplicity, we treat the case $T\equiv \text{id}.$ Let $(X,\g Y)$ be a random variable with probability distribution $\mu$ on $[1,\infty)\times \mathbb R^{d-1}.$ We call $\bar F(x,\g y)=\mu\{[x,\infty)\times (-\g \infty,\g y]\}$ the {\em $x$-survival function\/}. Subscripts $X$ or $\g Y$ refer to the corresponding marginal distribution. Suppose that $\bar F_X>0.$ We are interested in the weak limit \cite{Billingsley1995} of $(t^{-1} X,\g Y)\mid X\geq t,$ equivalent to the weak limit of its distribution $\mu(t A,B)/\bar F_X(t),$ and of its $x$-survival function $\bar F(tx,\g y)/\bar F_X(t).$ Let $H$ and $\nu$ be two probability distributions on $\mathbb R^{d-1}$ and $E$ respectively such that $\nu_X \neq \delta_1.$ We denote by $P_\alpha$ the Pareto distribution on $[1,\infty)$ with shape $\alpha>0.$ 
 
\begin{theorem}[One-Component Regular Variation for Distributions]\label{thm:equivRVxMeas}
The following are equivalent.
\begin{enumerate}[(i)]
\item \label{it:leRVxMeas_limit}  $(t^{-1}X,\g Y)\mid X\geq t \overset{w}{\Rightarrow} \nu,$
\item \label{it:leRVxMeas_RV_X} $\g Y\mid X\geq t \overset{w}{\Rightarrow} H$ and $t^{-1}X\mid X\geq t\overset{w}{\Rightarrow} P_\alpha,$ 
\item \label{it:leRVxMeas_ReprF} $\bar F(x,\g y)=c(x) \exp\left\{-\int_{e_0}^x \alpha(z)z^{-1}dz\right\}\,Q(\g y\mid x),$ for measurable $\alpha(t)\rightarrow \alpha,$ $c(t)\rightarrow c$ and a conditional cumulative distribution function $Q(\g y \mid t)\overset{w}{\Rightarrow} H,$ 
\item \label{it:leRVxMeas_RV} $(t^{-1}X,\g Y)\mid X\geq t \overset{w}{\Rightarrow} P_{\alpha}\,\times \,H,$
\end{enumerate}
\end{theorem}

We say that $(X,\g Y)$ is {\em regularly varying w.r.t.\ its first component}, written $(X,\g Y) \in \text{RV}_{-\alpha}^x(H),$ if (\ref{it:leRVxMeas_limit})--(\ref{it:leRVxMeas_RV}) are satisfied. To see that $\nu_X\neq \delta_1$ is a necessary assumption, consider for instance $\bar F(x)=e^{-x}.$ The equivalence of (\ref{it:leRVxMeas_limit}) and (\ref{it:leRVxMeas_RV}) can be found in \cite{Lindskog2014} (Theorem~3.1 with $\{\lambda,(x,\g y)\}\mapsto (\lambda x,\g y))$ or in \cite{Heffernan2007} (Proposition~2 with $\alpha\equiv 1,$ $\beta\equiv 0$). Our contribution is the representation in (\ref{it:leRVxMeas_ReprF}) and a short proof.

\begin{proof} 
(\ref{it:leRVxMeas_limit}) $\Rightarrow$ (\ref{it:leRVxMeas_RV_X}): By Theorem~2.1.4 in \cite{BasrakPhD2000}, $t^{-1}X \mid X\geq t \overset{w}{\Rightarrow} \nu_X$ implies $\nu_X=P_{-\alpha}$ for some $\alpha>0.$ (\ref{it:leRVxMeas_RV_X}) $\Rightarrow$ (\ref{it:leRVxMeas_ReprF}): since $\bar F_X\in \text{RV}_{-\alpha}$ and $\bar F_{\g Y\mid X\geq t}\overset{w}{\Rightarrow} H,$ write $\bar F(x,\g y)=\bar F_X(x)\bar F_{\g Y\mid X\geq x}(\g y\mid x)$ and apply the representation theorem on $\bar F_X.$ (\ref{it:leRVxMeas_ReprF}) $\Rightarrow$ (\ref{it:leRVxMeas_RV}): 
$${\bar F(tx,\g y)\over \bar F_X(t)}=  {c(tx)\over c(t)}\exp\left\{-\int_t^{tx} \alpha(z)z^{-1}dz\right\}\,Q(\g y\mid tx)\overset{w}{\Rightarrow} x^{-\alpha}H(\g y),$$
because $\int_t^{tx} \alpha(z)z^{-1}dz=\int_1^{x} \alpha(t \tilde z)\tilde z^{-1}dz\rightarrow \alpha \log x.$  (\ref{it:leRVxMeas_RV}) $\Rightarrow$ (\ref{it:leRVxMeas_limit}): clear.
\end{proof}

Suppose further that $(X,\g Y)$ admits a probability density $f$ w.r.t.\ the Lebesgue measure such that $f_X$ and $f(\cdot,\g 1)$ are positive. We want to know what the relation is between regular variation of $f$ and of $(X,\g Y).$ 

We start by answering the question in the univariate case. From Karamata's theorem (Theorem~\ref{thm:karamataRVx}), $f_X\in \text{RV}_{-\alpha-1}$ implies $\bar F_X\in \text{RV}_{-\alpha}$ and its representation has coefficient $c(\cdot)\equiv 1$. Thus, $X\in \text{RV}_{-\alpha}$ and its representation now determines exactly a probability distribution: $\bar F_X(x)= \exp\left\{-\int_1^x \alpha(z)z^{-1}dz\right\},$ $\forall x\geq 1,$ for a positive and measurable $\alpha(t) \rightarrow$ $ \alpha>0.$ The following example shows that regular variation of $X$ does not imply regular variation of $f_{X}$ in general. Equalities and pointwise convergences between densities hold almost everywhere.

\begin{example}[Regular Variation of the Distribution but not of the Density]\label{ex:XRVbutnotfRV}
Let $\bar F_X(x)=\exp\left\{-\int_1^x \alpha(y) y^{-1}dy\right\}$ on $[1,\infty)$ and $\alpha(x)=\sin(x)+2.$ Since $\bar F_X(x)= c(x)x^{-2}$ for $c(x)=\exp\{-\int_1^x y^{-1} \sin y \,dy\}$ satisfying $c(t)\rightarrow 1,$ the representation theorem gives $\bar F_X\in\text{RV}_{-2},$ and thus $X\in\text{RV}_{-2}.$ However, $tf_X(t)/ \bar F_X(t)=\alpha(t)\not\rightarrow 2,$ and from Karamata's theorem $f_X$ is not regularly varying. 
\end{example}

We now provide an answer in the multivariate case. If $f\in \text{RV}^x(v)$ such that the limit $v$ is non-null and the sequence is dominated by an integrable function, then by the dominated convergence theorem
\begin{align}\label{eq:convDensityMonDom}
f_{t^{-1}X,\g Y\mid X\geq t}(x,\g y) = c_t^{-1} {f(tx,\g y)\over f(t,\g 1)} \rightarrow c^{-1} v(x,\g y),
\end{align}
where $c_t=\int_1^\infty \int_{\mathbb R^{d-1}} {f(t x,\g y)/ f(t,\g 1)}d x d\g y
\rightarrow c=  \int_1^\infty \int_{\mathbb R^{d-1}} v( x,\g y)d x d\g y<\infty.$  Alternatively, we can assume $v$ integrable and the sequence monotone instead of dominated. Let $\alpha>0$ and $H$ be a probability distribution on $\mathbb R^{d-1}$ with density $h$ satisfying $h(\g 1)>0.$ According to Lemma~\ref{le:equivRVx}, the limit in (\ref{eq:convDensityMonDom}) has the form $\alpha x^{-\alpha-1}h(\g y),$ and thus $(X,\g Y)\in \text{RV}_{-\alpha}(H).$ This is an important result: we can guarantee regular variation of $(X,\g Y)$ simply by computing the limit of $f(tx,\g y)/f(t,\g 1),$ a useful approach when $\bar F$ is intractable or the weak convergences in Theorem~\ref{thm:equivRVxMeas} are difficult to check. Whereas monotonicity or domination is sufficient to obtain (\ref{eq:convDensityMonDom}), our next result reveals necessary and sufficient conditions.

 \begin{theorem}[One-Component Regular Variation for Densities]\label{thm:RVxForDens}
The following are equivalent.
\begin{enumerate}[(i)]
\item \label{it:corRVx_RVx} $f_{\g Y\mid  X}(\g 1\mid t)\rightarrow h(\g 1)$ and $f\in \text{RV}^x_{-\alpha-1}\{h(\cdot)/h(\g 1)\},$ 
\item \label{it:corRVx_RV_fX} $f_{\g Y\mid  X}(\g y\mid t)\rightarrow h(\g y)$ and $f_{X}\in \text{RV}_{-\alpha-1},$
\item \label{it:corRVx_RV_fcond1} $f_{\g Y\mid  X}(\g y\mid t)\rightarrow h(\g y)$ and $f_{X\mid \g Y=\g 1}\in \text{RV}_{-\alpha-1},$
\item \label{it:corRVx_RV_fcondy} $f_{\g Y\mid  X}(\g y\mid t)\rightarrow h(\g y)$ and $f_{X\mid \g Y=\g y}\in \text{RV}_{-\alpha-1},$ $\forall \g y$ s.t.\ $h(\g y)>0,$ 
\item \label{it:corRVx_repr_ax} $f(x,\g y)=\alpha(x) x^{-1}\exp\left\{-\int_1^x \alpha(z)z^{-1}dz\right\}\,q(\g y\mid x),$ for a positive and measurable $\alpha(t)\rightarrow \alpha$ and a conditional probability density $q(\g y\mid t)\rightarrow h(\g y),$ $\forall x\geq 1,$ $\forall \g y,$ 
\item \label{it:pointTotConv} $f_{t^{-1}X,\g Y\mid X\geq t}(x,\g y)\rightarrow \alpha x^{-\alpha-1} h(\g y).$
\end{enumerate} 
In this case, $(X,\g Y)\in \text{RV}^x_{-\alpha}(H).$
\end{theorem}
In (\ref{it:corRVx_RVx}), the first condition is necessary and can be replaced by $tf(t,\g 1)/\bar F(t)\rightarrow \alpha h(1),$ and in (\ref{it:corRVx_RV_fX}) by $t f(t,\g y)/\bar F_X(t)$ $\rightarrow$ $\alpha h(\g y)/h(\g 1).$ 

\begin{proof} (\ref{it:corRVx_RVx}) $\Leftrightarrow$ (\ref{it:corRVx_RV_fcondy}): straightforward; use Lemma \ref{le:equivRVx} for (\ref{it:corRVx_RV_fcond1}) $\Rightarrow$ (\ref{it:corRVx_RV_fcondy}) and (\ref{it:corRVx_RV_fcondy}) $\Rightarrow$ (\ref{it:corRVx_RVx}). (\ref{it:corRVx_RV_fX}) $\Rightarrow$ (\ref{it:corRVx_repr_ax}): write $f(x,\g y)=f_X(x)f_{\g Y\mid X}(\g y\mid x)$. On the one hand, $f_X\in \text{RV}_{-\alpha}$ and thus from Karamata's theorem (Theorem~\ref{thm:karamataRVx}) $\bar F_X$ has a representation with $c(\cdot)\equiv 1.$ Differentiate this representation to find one for $f_X.$ On the other hand, $f_{\g Y\mid X}(\g y\mid t)\rightarrow h(\g y)$ by assumption. Altogether, this gives (\ref{it:corRVx_repr_ax}), a well-defined probability density because $\epsilon$ and $q$ are measurable, and Tonelli's theorem ensures it integrates to $1.$ (\ref{it:corRVx_repr_ax}) $\Rightarrow$ (\ref{it:pointTotConv}): integrate (\ref{it:corRVx_repr_ax}) to find $\bar F(x)=\exp\left\{-\int_1^x \alpha(z) z^{-1}dz\right\},$ and thus
${tf(tx,\g y)/ \bar F_X(t)}=\alpha(tx)x^{-1} \exp\left\{-\int_{1}^{x} \alpha(tz)z^{-1}dz\right\}q(\g y\mid t x)\rightarrow \alpha x^{-\alpha-1} h(\g y).$ (\ref{it:pointTotConv}) $\Rightarrow$ (\ref{it:corRVx_RVx}): first,
$${f(tx,\g y)\over f(t,\g 1)}={tf(tx,\g y)\over \bar F_X(t)}{\bar F_X(t)\over tf(t,\g 1)}\rightarrow  x^{-\alpha-1}{h(\g y)\over h(\g 1)}.$$
Second, apply Scheffé's Lemma \cite{Durrett2010} on (\ref{it:pointTotConv}) to obtain $(X,\g Y)\in \text{RV}^x_{-\alpha}(H)$ and, in particular, $X\in \text{RV}_{-\alpha}.$ Karamata's theorem ensures that
$$f_{\g Y\mid X}(\g 1\mid t)={tf(t,\g 1)\over\bar F_X(t)}\;{\bar F_X(t)\over tf_X(t)}\rightarrow h(\g 1).$$
$ $
\end{proof}

The following example illustrates the use of Theorem~\ref{thm:RVxForDens}, which holds also when the convergence is not uniform in $\g y.$

\begin{example}[Illustration of Theorem \ref{thm:RVxForDens}]\label{ex:fRV_noUnifConv}
Consider 
 \begin{align*}
 f(x,y) = \left\{ 
    \begin{array}{ll} 
\frac{6}{5} x^{-2}  & \mbox{if } y\in(0,x^{-1}),\\
 \frac{6}{5} x^{-2} y  & \mbox{else},\\
    \end{array}
\right.
 \quad 
 \text{on } \, [1,\infty)\times (0,1].
\end{align*}
The convergence ${f(tx,y)/ f(t,1)}\rightarrow x^{-2} y$ gives $f\in \text{RV}_{-2}\{h(\cdot)/h(1)\}$ for the cdf $H(y)=y^2$ with probability density $h(y)=2y$ on $(0,1].$ Since $f(t,1)/f_X(t)\rightarrow 2=h(1)$ (condition (\ref{it:corRVx_RVx}), Theorem~\ref{thm:RVxForDens}), $(X,\g Y)\in \text{RV}_{-\alpha}(H),$ i.e., ${\bar F(tx,y)/ \bar F_X(t)}\rightarrow x^{-1}y^2.$ The same conclusion is drawn by showing $f_X\in \text{RV}_{-2}$ and $f_{Y\mid X}(y\mid t)\rightarrow 2y = h(y)$ (condition (\ref{it:corRVx_RV_fX}), Theorem~\ref{thm:RVxForDens}).
\end{example}

 The results of this section are easily extended to hold for the arbitrary decay $T$ defined in (\ref{eq:defStar}). We write $(X,\g Y)\in T\text{-RV}^x_{-\alpha}(H)$ if $\{T(X),\g Y\}\in \text{RV}^x_{-\alpha}(H).$ From the continuous mapping theorem, this is equivalent to $\mu_{t\star}\overset{w}{\Rightarrow}\nu,$
where $\mu_{\lambda\star}(A,B):={\mu(\lambda\star A,B)/ \mu_X(\lambda \star E)}$ is a probability distribution on $E\times \mathbb R^d$ for every $\lambda\in E.$ In this case, $\nu=P_\alpha \circ T\times H,$ and in particular $\nu_X$ has survival function $T(x)^{-\alpha}.$ We say that $\mu$ is homogeneous w.r.t.\ its first component if $\mu(\lambda \star A, B)=T(\lambda)^{-\alpha} \mu(A, B),$ $\forall \lambda \in E.$ Theorem~\ref{thm:equivRVxMeas} characterizes such distributions. 

\begin{corollary}[Homogeneous Distributions]\label{cor:charHomgOneComp} The following are equivalent: $\mu$ is homogeneous of order $-\alpha$ w.r.t.\ its first component; $\bar F$ is homogeneous of order $-\alpha$ w.r.t.\ its first component; $\mu_{\lambda \star}=\mu,$ $\forall \lambda \in E;$ $X\ci \g Y$ and $X\sim P_\alpha\circ T;$ (under its existence) $f_{X, \g Y}$ is homogeneous of order $-\alpha-1$ w.r.t.\ its first component.
\end{corollary}

We now focus on one-component regular variation when the first component is the radius of a random vector $\g X$ and the rest of the vector is the angle  --- note the change in the notation. Let $(R,\g \Theta)=$ $\phi(\g X)$ $=\{r(\g X),\g\theta(\g X)\}$ be its expression in radial coordinates defined in (\ref{eq:defnRadialCoord}) with $T\equiv \text{id}$ and suppose that $\bar F_R>0$. Let $\alpha>0$ and $\g Y\sim \nu,$ a probability distribution on $C_r=\{ r(\g x)\geq 1\},$ such that $r(\g Y)\not \sim \delta_1.$ 

Analogously to functions, $\g X$ is multivariate regularly varying in the sense (\ref{eq:mrvVector}) if and only if $(R,\g \Theta)$ is one-component regularly varying. This relation is well-known when, for instance, $\phi$ is the change of variables into pseudo polar coordinates \cite{basrak2002characterization}.

\begin{proposition}[Multivariate and One-Component Regular Variation]\label{pro:qRVoRV_a} 
\begin{align}\label{eq:qRVoRV_a} 
t^{-1}\g X\mid \g X\in tC_r \overset{w}{\Rightarrow}  \nu \; \text{ on }\; C_r \quad \iff \quad (R,\g\Theta)\in \text{RV}_{-\alpha}(H)\;\text{ on }\;[1,\infty)\times \Omega.
\end{align}
In this case, $\nu$ is homogeneous of order $-\alpha,$ i.e., $ \nu(\lambda A)=\lambda^{-\alpha}\nu(A),$ $\forall \lambda\geq 1.$
\end{proposition}

\begin{proof}
We only prove the direct implication; the reverse goes similarly. Since $\phi^{-1}$ is homogeneous of order $1$ w.r.t.\ its first component, write $t^{-1}R,\g \Theta\mid R\geq t=\phi(t^{-1}\g X)\mid R\geq t,$ which converges weakly to a probability distribution by the continuous mapping theorem. Hence, $(R,\g\Theta)\in \text{RV}_{-\alpha}(H).$ To show the homogeneity of $\nu,$ let $\g Y=\phi^{-1}(R^\star,\g\Theta^\star) \sim \nu$ and use Corollary~\ref{cor:charHomgOneComp}:
\begin{align*}
\nu(\lambda A)&=\pr\{\phi^{-1}(\lambda^{-1}R^\star,\g\Theta^\star)\in A\mid R^\star\geq \lambda\}\pr(R^\star\geq \lambda)\\
&=\lambda^{-\alpha} \pr\{\phi^{-1}(R^\star,\g\Theta^\star)\in A\}  = \lambda^{-\alpha}\nu(A),\quad\forall \lambda\geq 1.
\end{align*}
\end{proof}

 We denote the multivariate regular variation in (\ref{eq:qRVoRV_a}) by $\g X\in \text{RV}_{-\alpha}(\nu)$ or $\text{RV}_{-\alpha}(\g Y).$ Regular variation w.r.t.\ another radial function $\tilde r(\cdot)$ is equivalent if there exists $c_1,c_2>0$ such that $c_1 r(\g x)\leq \tilde r(\g x)\leq c_2 r(\g x),$ $\forall \g x,$ and in this case, the limits are related as follows: $\nu_r(A)=\nu_{\tilde r}(c_2 A)/\nu_{\tilde r}(c_2 C_r).$ 

Suppose that $\g X$ admits a probability density $f_{\g X}$ w.r.t.\ the Lebesgue measure and let $f_{\g Y}$ be a probability on $C_r$ such that $f_{\g X}(\lambda \g 1),f_{\g Y}(\lambda \g 1)>0,$ $\forall \lambda\geq 1$. Similarly to what we did in the previous section, we now study the relation between regular variation of $f_{\g X}$ in the sense (\ref{eq:TRV}) and of $\g X.$ If $f_{\g X}\in \text{RV}(v)$ for some non-null function $v$ and the sequence is dominated by an integrable function, then
\begin{align}\label{eq:convDensityfXMonDom}
f_{t^{-1}\g X\mid \g X\in  tC_r}(\g x) = {f_{\g X}(t \g x)\over f_{\g X}(t \g 1)} \left( \int_{C_r} {f_{\g X}(t \tilde{\g x})\over f_{\g X}(t\g 1)}d\tilde{\g x} \right)^{-1} \rightarrow {v(\g x)\over \int_{C_r} v(\tilde{\g x})d\tilde{\g x} }=:f_{\g Y}(\g x),
\end{align}
thus the limiting probability density $f_{\g Y}$ is homogeneous and by Scheffé's Lemma $\g X\in \text{RV}_{-\alpha}(\g Y)$ on $C_r.$ As a comparison, Proposition 5.20 in \cite{Resnick1987Extreme} has the same conclusion, however, its assumptions require the convergence to be uniform and the limiting density to be bounded on $\{\g x\in \mathbb R^d: r(\g x)=1\}.$ 

Relations (\ref{eq:qRVoRV_a}) and (\ref{eq:convDensityfXMonDom}) are easily extended in the case of the general decay $T.$ We write $\g X\in T\text{-RV}^x_{-\alpha}(H)$ if $\{T(X_1),\ldots, T(X_d)\}\in \text{RV}^x_{-\alpha}(\nu)$ or equivalently $\mu_{t\star}\overset{w}{\Rightarrow}\nu,$ where $\mu_{\lambda\star}(A):=\mu(\lambda\star A)/ \mu(\lambda \star C_r)$ defines a probability distribution on $C_r,$ $\forall \lambda\in E.$ In this case, $\nu$ is homogeneous of order $-\alpha,$ i.e.,
$\mu(\lambda \star A)=T(\lambda)^{-\alpha} \mu(A),$ $\forall \lambda \in E.$


\section{Multivariate Exceedances}\label{par:me} 
Extreme-value theory aims at describing the distribution of an $\mathbb R^d$-random vector $\g X$ when at least one of its marginal is large. Typically, it is assumed that $\g X$ is regularly varying on $C_{||\cdot||_\infty}=\{\g x\in\mathbb R^d:||\g x||_\infty\},$ 
i.e., $t^{-1}\g X\mid ||\g X||_\infty\geq t \overset{w}{\Rightarrow} \g Y,$ and that $\g Y$ admits a probability density $f_{\g Y},$ known to be homogeneous. This excludes the case of asymptotically independent marginals \cite{Resnick1987Extreme}, which is to be treated separately. The tail distribution of $\g X$ is commonly inferred by censoring the marginals whose absolute values fall under some large threshold \cite{Ledford1996}, requiring a model for $f_{\g Y}$ and its censored versions. In this section, we first suggest a general procedure to obtain parametric homogeneous densities whose multivariate marginals and censored densities are known. Second, we explain how to simplify $f_{\g Y}$ when $d$ is large using graphical models.  

We start by asking wether the marginals of a regularly varying vector are also regularly varying and, in this case, if the marginals of the limit are the limits of the marginals. The following lemma provides an answer. Let $A\subseteq \{1,\ldots,d\},$ $C^{A}_{||\cdot||_\infty}:=\{\g x_A\in \mathbb R^{|A|}:||\g x_A||_\infty\geq 1\}$ and $v\not \equiv 0.$
\begin{lemma}[Regular Variation of the Marginals]\label{le:marginalsRV}
Suppose that 
\begin{align}\label{eq:hyp_xmaxc}
\pr(b X_i\geq t\mid \,||\g X||_\infty\geq t) \rightarrow c_{b i}>0,\quad b\in \{-1,1\},\; i=1,\ldots, d.
\end{align}
 If $\g X\in\text{RV}_{-\alpha}(\g Y)$ on $C_{||\cdot||_\infty},$ then $\g X_A\in \text{RV}_{-\alpha}(\g Y_A\mid \,||\g Y_A||_\infty \geq 1)$ on $C^A_{||\cdot||_\infty}.$ If $f_{\g X} \in \text{RV}_{-\alpha}(v)$ and the sequence is either monotone and $v$ is integrable or is dominated by an integrable function, then $f_{t^{-1} \g X_A\mid\, ||\g X_A||_\infty \geq t} \rightarrow f_{\g Y_A\mid \,||\g Y_A||_\infty\geq 1}.$
\end{lemma}
\begin{proof}
Under (\ref{eq:hyp_xmaxc}), $\pr(||\g X_A||_\infty \geq t\mid \,||\g X||_\infty \geq t )\rightarrow c>0.$ For all $\nu$-continuous Borel set $B,$
$${\pr(\g X_A\in t B)\over \pr(||\g X_A||_\infty\geq t)}= {\pr(\g X_A\in t B)\over \pr(||\g X||_\infty\geq t)}{\pr(||\g X||_\infty\geq t)\over 
\pr\{\g X\in t(C^A_{||\cdot ||_\infty}\times \mathbb R^{|A^C|})\}},$$ and converges to $\pr(\g Y_A\in B\mid \,||\g Y_A||_\infty \geq 1),$ proving the first part of the lemma. The second part is a direct consequence of the monotone or dominated convergence theorem.
\end{proof}

This result is specific to the Pareto decay ($T\equiv \text{id}$): $\bar F(x,y)=2/(e^{|x|}+e^{|y|})$ is homogeneous on $\mathbb R^2$ w.r.t.\ addition but its marginal $\bar F_X(x)=2/(e^{|x|}+1),$ although regularly varying, is not homogeneous.

\begin{example}[Limiting Density of a Multivariate Student Distribution]\label{ref:studentLimit}
Let $\g X$ be multivariate Student $t$-distributed with mean $\g 0$ and dispersion matrix $\Sigma\subset \mathbb R^{d\times d}.$ The sequence $f_{\g X}(t\g x)/ f_{\g X}(t\g 1)$ is monotone and converges to $v(\g x)/v(\g 1)$ for $$v(\g x)=(\g x^T \Sigma \g x)^{-(\nu+d)/2},$$ which integrates to $c<\infty$ on $C_{||\cdot||_\infty}.$ As explained in (\ref{eq:convDensityfXMonDom}), $f_{t^{-1}\g X\mid \,||\g X||_\infty\geq t}(\g y)\rightarrow c^{-1} v(\g y)=: f_{\g Y}(\g y),$ thus $\g X\in\text{RV}_{-\alpha}(\g Y).$ In addition, the censored limiting density is the limit of the censored density, that is $f_{t^{-1}\g X_A, |\g X_A|<t\mid \,||\g X||_\infty \geq t}\rightarrow f_{\g Y_A,|\g Y_A|<1},$ and Lemma~\ref{le:marginalsRV} ensures that the marginals of the limit are the limits of the marginals, i.e., $f_{t^{-1}\g X_A\mid ||\g X_A||_\infty \geq t}\rightarrow f_{\g Y_A}.$ These quantities can thus be computed from the censored densities and marginals of the multivariate Student by passing to the limit. Up to a constant and a transformation of the univariate marginals, $f_{\g Y}$ coincides with the density of the exponent measure of the extremal-$t$ distribution derived by \cite{ribatet2013spatial}. 
\end{example}

As the previous example suggests, we can obtain parametric families for homogeneous distributions by computing limits of well-known multivariate distributions that are regularly varying with decay $T.$ We now focus on a complementary approach for modeling $f_{\g Y}$ using graphical models. In short, graphical models offer a way to simplify a joint density by assuming some conditional independence between the marginals. Conditional independence --- and dependence --- between marginals is only meaningful on a product space \cite{Dawid2001}. Since $\g Y$ has values in $C_{||\cdot||_\infty},$ the range of $Y_i$ depends on the other marginals.  To remedy this problem, we work with the censored vector 
\begin{align}
\g Y^C= \left\{
 \begin{array}{ll}
(Y_1 1_{|Y_1|\geq 1},\ldots, Y_d 1_{|Y_d|\geq 1}) &\text{ with probability } p,\\
\g 0 & \text{ with probability } 1-p,
\end{array}
\right.
\end{align}
whose values lie in $F^d=\{(\infty,-1]\times \{0\}\times [1,\infty)\}^d,$ where $p=\pr(||\g X||_\infty\geq u)\in (0,1)$ for some large threshold $u.$ This stays in line with statistical inference, which typically estimates a censored version of the limiting distribution.

The Hammersley-Clifford theorem states that if $f_{\g Y^C}$ is positive, it factorizes w.r.t.\ a graph $\mathcal G=(\{1,\ldots,d\},E)$ if and only if $Y_i^C \ci Y_j^C  \mid \g Y^C _{\{1,\ldots,d\}\setminus\{i,j\}},$ $\forall (i,j)\notin E$ \cite{Lauritzen}. The latter denotes conditional independence between $Y^C_i$ and $Y^C_j$ given the rest of the vector. Let us illustrate an application of the theorem. 
 
\begin{example}[Factorization of the Limiting Density]\label{ex:CI3hom}
Let $\g Y$ have values in $\mathbb R_+^3$ with $f_{\g Y_{12}\,\mid\, ||\g Y_{12}||_\infty\geq 1}(x,y)=\frac{4}{3}(x+y)^{-3}=f_{\g Y_{23}\,\mid \,||\g Y_{23}||_\infty\geq 1}(x,y)$ and $p=\pr(||\g Y_{12}||_\infty\geq 1).$ From the Hammersley-Clifford theorem, $Y^C_1  \ci Y^C_3 \mid  Y^C_2$ if and only if
$$f_{\g  Y^C}(x,y,z)={f_{\g Y^C_{12}}(x,y)f_{\g Y^C_{23}}(y,z)\over f_{Y^C_{2}}(y)},\quad  \forall (x,y,z)\in (\{0\}\cup [1,\infty))^3,$$
where $f_{\g Y^C_{12}}(x,0)=\frac{2}{3} p \{x^{-2}-(x+1)^{-2}\},$ $f_{Y^C_2}(x)=\frac{2}{3} p x^{-2},$ etc.
\end{example}

From a statistical perspective, the distribution of $\g Y^C$ is to be estimated from samples of $\g X.$ To that end, we consider a sequence $\g X^C_t$ which is function of $\g X$ and satisfies $\g X^C_t\overset{w}{\Rightarrow} \g Y^C$ if $\g X\in \text{RV}_{-\alpha}(\g Y).$ Crucially, the sequence of marginals $(\g X_{t}^C)_D$ for $D\subseteq \{1,\ldots,d\}$ coincides with the censored sequence $(\g X_{t,D})^C$ converging to $\g Y_{t,D}^C.$ It is defined for all $\g x\in \mathbb R^{|D\cap A|}$ s.t.\ $|\g x|\geq 1,$ $A\subseteq D\setminus \emptyset,$ by 
\begin{multline}\label{eq:YC}
 \pr( \g \sigma_{\g x} \g X^C_{t,D\cap A}\geq |\g x|, \g X^C_{t,D\cap A^c}=\g 0)   =
     p\, {\pr(\g \sigma_{\g x}  \g X_{D\cap A}\geq t |\g x|, |\g X_{D\cap A^C}|< t )\over  \pr(||\g X||_\infty \geq t)},
     \end{multline}
and $\pr( \g X^C_{t,D}=\g 0)    =  1-p \pr(||\g X_D||_\infty \geq t\mid ||\g X||_\infty \geq t),$ where $\g \sigma_{\g x}$ denotes the sign of $\g x.$ Both $\g X_t^C$ and $\g Y^C$ admit probability densities $f_{\g X_t^C}$ and $f_{\g Y^C}$ w.r.t.\ the measure $\mu_0(A_1\times \ldots\times A_d):=\sum_{D\subseteq \{1,\ldots,d\}}\lambda^{|D|}\left(\prod_{i\in D} A_i\right) \delta_{\g 0}^{|D^c|}\left(\prod_{i\in D^c} A_i\right)$
on $F^d,$ where $\lambda^d$ is the $d$-dimensional Lebesgue measure and $\delta^d_{\g 0}$ the Dirac measure on $\g 0\in\mathbb R^d.$ 

We want to find conditions on $\g X$ that impose a factorization of $f_{\g Y^C}.$ This brings us to extend the Hammersley-Clifford theorem for densities that factorize in the limit. Conditional independence is translated into {\em asymptotic conditional independence\/}, a notion little mentioned in the literature. 


\section{Asymptotic Graphical Models}\label{par:agm}
Let $\g X_t$ for $t\geq 1$ and $\g X$ be $\mathbb R^d$-valued random vectors with almost everywhere (a.e.) continuous probability densities $f_{\g X_t}$ and $f_{\g X}$ w.r.t.\ a base measure $\mu_0,$ typically the Lebesgue, the counting measure, or a combination of the two. For disjoint subsets $A,B,C\subseteq \{1,\ldots,d\},$ we call the marginals $\g X_{t,\,A}$ and $\g X_{t,\,B}$ {\em asymptotically conditionally independent\/} with respect to $\g X_{C,\,t},$ written $\g X_{t,\,A}\,\aci\, \g X_{t,\,B} \mid \g X_{t,\,C},$ if
\begin{align}\label{def:asyCondIndep}
  (f_{\g X_{t,\,ABC}} \,f_{\g X_{t,\,C}} -f_{\g X_{t,\,AC}}\,f_{\g X_{t,\,BC}} )d\mu_0
\overset{w}{\Rightarrow} 0.
\end{align}
Here, $f_{\g X_t}d\mu_0 \overset{w}{\Rightarrow}$ stands for weak convergence of $\g X_t,$ i.e., the convergence of $\int g f_{\g X_t} d\mu_0$ for every $g$ a.e.\ continuous and bounded. When $\g X_t=\g X_1,$ $\forall t,$ (\ref{def:asyCondIndep}) coincides with conditional independence of random variables \cite{Lauritzen}. 

Let $\mathcal G=(V,E)$ be an undirected graph with set of nodes $V=\{1,\ldots,d\}$ and set of edges $E.$ We say that $\g X_t$ satisfies the {\em asymptotic Markov properties\/} according to $\mathcal G$ if $\{i,j\}\notin E$ implies $X_{t,i} \aci X_{t,j}\mid \g X_{t,\,V \setminus \{i,j\}}.$ Moreover, we say that $f_{\g X_t}$ {\em asymptotically factorizes\/} w.r.t.\ $\mathcal G$ if $\g X_t \overset{w}{\Rightarrow}\g X$ and $f_{\g X}$ factorizes w.r.t.\ $\mathcal G.$

\begin{theorem}[Asymptotic Hammersley-Clifford Theorem]\label{thm:equivAGM_weak}
Suppose $\g X_t \overset{w}{\Rightarrow} \g X,$ $f_{\g X}>0$ and $f_{\g X_A}$ bounded $\forall A\subseteq V$ a.e. The error $\epsilon_t:=f_{\g X_t}-f_{\g X}$ satisfies 
\begin{align}\label{eq:errorTermHyp}
(\epsilon_{t,V} \epsilon_{t,V\setminus\{i,j\}}- \epsilon_{t,V\setminus\{i\}}\,\epsilon_{t,V\setminus\{j\}})d\mu_0 \;\overset{w}{\Rightarrow} 0,\quad \forall (i,j)\in E,
\end{align}
if and only if the following are equivalent:
        \begin{enumerate}[(i)]
\item \label{it:af} $f_{\g X_t}$ asymptotically factorizes w.r.t.\ $\mathcal G,$
 \item \label{it:gm} $\g X$ satisfies the Markov properties according to $\mathcal G,$
\item \label{it:agm} $\g X_t$ satisfies the asymptotic Markov properties according to $\mathcal G.$
 \end{enumerate}
 \end{theorem}

\begin{proof}
(\ref{it:af}) $\Leftrightarrow$ (\ref{it:gm}): apply Hammersley and Clifford theorem. (\ref{it:gm}) $\Leftrightarrow$ (\ref{it:agm}): it suffices to show relations of the form $f_{X Y\g Z}f_{\g Z}=f_{X\g Z} f_{Y\g Z}$ a.e.\ if and only if $(f_{X_t Y_t \g Z_t}f_{\g Z_t}-f_{X_t \g Z_t} f_{Y_t \g Z_t})d\mu_0 \overset{w}{\Rightarrow} 
 0.$ Rewrite the latter as
\begin{align}\label{eq:weakeqerror} 
(f_{X Y\g Z}f_{\g Z}-f_{X\g Z} f_{Y\g Z})d\mu_0 + (\epsilon_{X_t,Y_t,\g Z_t}\epsilon_{\g Z_t}- \epsilon_{X_t,\g Z_t}\epsilon_{Y_t,\g Z_t}
)d\mu_0 \\
+(\epsilon_{X_t,Y_t,\g Z_t} f_{\g Z}+\epsilon_{\g Z_t} f_{X,Y,\g Z}-\epsilon_{X_t,\g Z_t}f_{Y,\g Z}-\epsilon_{Y_t,\g Z_t}f_{X,\g Z})d\mu_0 \overset{w}{\Rightarrow} 0.\nonumber
\end{align}
 Since $(X_t,Y_t,\g Z_t)\overset{w}{\Rightarrow} (X,Y,\g Z),$ $\epsilon_t$ converges weakly to $0$ and so does the last term in (\ref{eq:weakeqerror}) because $f_{X,Y,\g Z}$ and its marginals are a.e.\ continuous and bounded. Moreover, the middle term vanishes by assumption (\ref{eq:errorTermHyp}), proving the equivalence.
  \end{proof}
  
 As mentioned in \cite{Lauritzen}, Example 3.11, the Markov property is in general not satisfied under weak convergence; Theorem~\ref{thm:equivAGM_weak} provides conditions under which it is. If $f_{\g X_t}\rightarrow f_{\g X}$ and the sequence is dominated by an integrable function, then $f_{\g X_{t,A}}\rightarrow f_{\g X_A}$ for all $A\subseteq \{1,\ldots,d\}.$ Hence, $f_{\g X}$ need not be bounded and, crucially, (\ref{eq:errorTermHyp}) holds directly. Similarly, if $f_{\g X_t}\rightarrow f_{\g X}$ in $L_2,$ then $f_{\g X_{t,A}}\rightarrow f_{\g X_A}$ in $L_2$ by Jensen's inequality, and we apply Hölder's inequality to show that (\ref{eq:errorTermHyp}) is satisfied.


\section{Asymptotic Graphical Modeling of Exceedances}\label{par:agme} 
We explain how to model the exceedances of a random vector $\g X$ with probability density $f_{\g X}$ using asymptotic graphical models. Let $\mathcal G$ be a decomposable graph with set of cliques $\mathcal C$ and set of intersections between them $\mathcal D.$ Suppose that $f_{\g X}\in\text{RV}_{-\alpha}(v)$ for $v \not \equiv 0$ and that the sequence is either monotone and $v$ is integrable or dominated by an integrable function. As seen in (\ref{eq:convDensityMonDom}), $f_{t^{-1}\g X\,\mid\,||\g X||_\infty \geq t}\rightarrow f_{\g Y}$ and we assume that the probability density $f_{\g Y}$ is positive. It follows that $f_{\g X^C_t}\rightarrow f_{\g Y^C}>0,$ where $\g X^C_t$ and $\g Y^C$ are defined in Section \ref{par:me}. From Theorem~\ref{thm:equivAGM_weak}, $\g X^C_t$ satisfies the asymptotic Markov properties according to $\mathcal G$ if and only if $f_{\g Y^C}$ factorizes w.r.t.\ $\mathcal G,$ i.e.,
\begin{align}\label{eq:fact_f}
f_{\g Y^C}(\g y)={\prod_{C\in \mathcal C} f_{\g Y^C_{C}}(\g y_{C})\over \prod_{D\in \mathcal D} f_{\g Y^C_{D}}(\g y_{D})}, \quad \g y\in F^d,
\end{align}
Hence, for all $A\subseteq \{1,\ldots,d\}$ and $\g y_A$ s.t.\ $|\g y_A|\geq \g 1,$
\begin{align}\label{eq:fy}
f_{\g Y_A, |\g Y_A^c|<  1}(\g y_A)={\prod_{C\in \mathcal C: C\cap A\neq \emptyset} f(\g y_{C\cap A}) \over \prod_{D\in\mathcal D: D\cap A\neq \emptyset} f(\g y_{D\cap A})} \; {\prod_{C\in\mathcal C } q_{C,A}\over \prod_{D\in\mathcal C } q_{D,A} },
\end{align}
where $f(\g y_{S\cap A})$ denotes the density of $\g Y_{S\cap A},\g Y_{S\cap A^c}<  1\mid \,||\g Y_S||_\infty \geq 1$ and
$$
q_{S,A}=\left\{
\begin{array}{ll}
p_S := p \,\pr(||\g Y_S||_\infty \geq 1) & \text{ if } S\cap A\neq \emptyset, \\
1-p_S & \text{ else, }\\
\end{array}
\right.
$$
for all $S\in \mathcal C\cup \mathcal D.$ By positivity of $f_{\g Y},$ (\ref{eq:hyp_xmaxc}) is satisfied and thus Lemma~\ref{le:marginalsRV} gives $\g X_S\in \text{RV}_{-\alpha}(\g Y_S\mid \,||\g Y_S||_\infty \geq 1)$ on $C_{||\cdot||_\infty}^S$ with censored limiting density $f(\g y_{S\cap A}).$ This is a key step as it means that the latter can be estimated from samples of $\g X_S.$ 

This suggests the following strategy to model exceedances of an i.i.d.\ sample of $\g X.$ First, transform the positive and negative part of each marginal that exceeds a large threshold to unit Pareto to obtain the censored vector $\g X^C_{1}.$ Second, select a decomposable graph $\mathcal G$ by testing conditional independence between the marginals of $\g X^C_{1}=$ $(X_1 1_{|X_1|\geq 1},\ldots, X_d 1_{|X_d|\geq 1}).$ Third, for all $S\in \mathcal C\cup \mathcal D,$ select a model for the homogeneous density $f_{\g Y_S\mid\, ||\g Y_S||\geq 1}$ with unit Pareto marginals; estimate it using a sample from $\g X^C_{1,S},$ and estimate $p_S\approx \pr(||\g X_{1,S}||_\infty\geq 1).$ This determines the asymptotic approximation of $f_{\g X^C_1}$ in (\ref{eq:fy}).

In low dimensions, there is a rich class of models for extreme-value distributions \cite{Segers2010}. If the later have Fréchet marginals, we have seen that their exponent measure density, when it exists, is homogeneous and thus provides a model for $f_{\g Y_S\,\mid\,||\g Y_S||_\infty\geq 1}.$ In a parametric approach, a refinement is to impose the parameters of $f_{\g Y_C}(\cdot;\g\theta_C)$ to be consistent with the ones of $f_{\g Y_D}(\cdot;\g\theta_D)$ whenever $D\subset C$ and to estimate them from sufficient statistics; Example~\ref{ref:studentLimit} showed a way to find consistent parametric families. An extreme-value distribution often used in practice is the Hüsler-Reiss \cite{Husler1989,Engelke2015}.

\begin{example}[Hüsler-Reiss and Conditional Independence]
The Hüsler-Reiss exponent measure has homogeneous probability density $f_{\g Y\mid Y_k\geq 1}(\g y)= y_k^{-2} \phi(\g y_{-k}\mid y_k)$ on $C_k=\{\g y\geq \g 0: y_k\geq 1\},$ $\forall k=1,\ldots,d,$ where $\phi$ is the density of the multivariate log-normal distribution on $\mathbb R^{d-1}$ with mean $(\log y_k-\frac{1}{2}\Gamma_{jk})_{j\neq k}$ and covariance matrix $\Sigma_\vartheta^{-1},$ where $\Sigma_\vartheta =\frac{1}{2} \{\Gamma_{ik}+\Gamma_{jk}-\Gamma_{ij}\}_{2\leq i,j\leq d},$ and $\Gamma_{ij}=\Sigma_{ii}+\Sigma_{jj}-2\Sigma_{i,j}$ is the incremental variance defined by a correlation matrix $\Sigma\subset \mathbb R^{d\times d}.$ The marginals $\g Y_A\mid \,||\g Y_A||_\infty \geq 1$ and the censored densities can be computed. Moreover, conditional independence corresponds to a specific constraint on the covariance matrix, analogously to the Gaussian case:
\begin{align}\label{eq:CI_HR}
Y_i\ci Y_j\mid \g Y_{\{1,\ldots,d\}\setminus \{i,j\}} \; \text{ on } C_k\quad  \iff \quad  \Sigma_{\vartheta,ij}=\Gamma_{ij}-\Gamma_{ik}-\Gamma_{jk}=0.
\end{align}
\end{example}

$ $

A factorization of the homogeneous limiting density induces a specific factorization of its angular density. As explained in the next example, this gives a way to build a high-dimensional angular density satisfying the marginal constraints.

\begin{example}[Factorization of the Angular Density]
In (\ref{eq:fy}), $f_{\g Y^C}$ is expressed in terms of homogeneous lower dimensional densities that can be written as
$$f_{\g Y_{ S}\mid\, ||\g Y_{S}||_\infty \geq 1}(\g y_{S})= k  r(\g y_{S})^{-\alpha-1}h_{S}\{\g \theta(\g y_{S})\}J_{\phi_{S}}(\g y_{S}),\quad S\in \mathcal C\cup \mathcal  D,$$
where $\phi_S:\g y_S\mapsto (r,\g\theta)$ is the radial system of coordinates defined in (\ref{eq:defnRadialCoord}) with Jacobian determinant $J_{\phi_S}>0,$ $h_S$ is the angular probability density on $\Omega_S,$ and $k$ the normalizing constant. When no marginal is censored, $f_{\g Y^C}$ is homogeneous and its angular density reads
\begin{align}\label{eq:fact_h}
h(\g \theta)={k\over J_\phi(\g\theta)}
 {\prod_{C\in \mathcal C} \; r(\g \theta_{C})^{-\alpha-1}  \;h_{C}(\g\theta_{C})\, J_{\phi_{C}}(\g \theta_{C})   \over \prod_{D\in \mathcal D}  \; r(\g \theta_{D})^{-\alpha-1} h_{D}(\g\theta_{D})\,  J_{\phi_{D}}(\g \theta_{D})}.
\end{align}
The censored case goes similarly. We now derive the constraint that a parametric family for angular densities must satisfy to be consistent. Consistency means that for all $\tilde A= A\cup \{i\}\subseteq \{1,\ldots,d\},$ 
 $$f_{\g Y_{A}\mid \,||\g Y_A||_\infty\geq 1}(\g y_{A})= k \int_{\mathbb R} f_{\g Y_{\tilde A}}(\g y_{\tilde A})dy_i,$$
and translates into 
\begin{align}\label{eq:h_consist}
h_{A}(\g\theta)= k \int_{\mathbb R} r\{\g \theta^{-1}_{A}(\g \theta), z\}^{-\alpha-1}\,h_{\tilde A}\left[\g\theta\{\g \theta^{-1}_{A}(\g \theta), z\}\right]\, {J_{\phi_{\tilde A}}\{\g \theta^{-1}_{A}(\g\theta),z\}\over  J_{\phi_{A}}\{\g \theta^{-1}_{A}(\g \theta)\}}d z,
\end{align}
 by a change of variable $z= y_i/r(\g y_{A})$ and setting $\g y= r(\g y_A) \theta^{-1}_{A}(\g \theta).$ In particular, when $r(\g x)=||\cdot||_1,$ $\g\theta(\g x)=\g x/r(\g x),$ and $\g x\geq \g 0,$ (\ref{eq:h_consist}) becomes 
$$h_{A}(\g \theta)= k \int_0^\infty (1-w)^{\alpha+d-2} h_{\tilde A}\{(1-w)\g\theta,w\}dw.$$
 using the change of variable $w=z/(z+1).$ \cite{BallaniSchlather} showed that the asymmetric logistic model satisfies this constraint and is thus a consistent family  --- its censored density is however unattractive. 
\end{example}

Asymptotic graphical modeling for extremes necessitates testing asymptotic conditional independence, which is difficult in a non-parametric setting. As a first approximation, we can easily test conditional independence of the binary vector $\g B=(1_{| X_{1}|\geq t},\ldots, 1_{| X_{d}|\geq t})$ or a more refined discretization \cite{Nagarajan2013}. Further efforts are needed to build a test for tail conditional independence in the parametric continuous case --- for instance, by relying upon (\ref{eq:CI_HR}). Lastly, we mention the success of the Gaussian graphical lasso \cite{Friedman2008}, which has been able to tackle high-dimensional problems by imposing sparsity of the inverse covariance matrix. Future work could allow a similar approach for the Hüssler-Reiss distribution.

\paragraph*{Acknowledgements}
The first author is grateful to Anthony Davison for suggesting graphical modeling of extreme-value distributions, Richard Davis who proposed the notion of asymptotic conditional independence, Thomas Mikosch for a stimulating lecture on regular variation, Phyllis Wan, Sebastian Engelke for helpful discussions and the anonymous referee for their pertinent comments. He would also like to thank the Marquise and Marquis de Amodio for their funding.

\bibliographystyle{abbrvnat}
\bibliography{orv_ref}
\end{document}